\documentclass[11pt]{gtpart}
\usepackage{pinlabel} %for algebraic and geometric topology
\usepackage{amsmath}
\usepackage{amsfonts}
\usepackage{amssymb}
\usepackage{amscd}
\usepackage{fancyvrb}
\usepackage[all]{xy}
\usepackage{tikz}
\usepackage{amsthm}
\usetikzlibrary{matrix}
\usepackage{graphicx}
\usepackage{epsfig}

\newtheorem{theorem}{Theorem}[section]
\newtheorem{proposition}[theorem]{Proposition}
\newtheorem{corollary}[theorem]{Corollary}
\newtheorem{lemma}[theorem]{Lemma}
\newtheorem{definition}[theorem]{Definition}

\newtheorem{remark}[theorem]{Remark}
\newtheorem{conjecture}[theorem]{Conjecture}

\newcommand{\ktheory}{$K$-theory }
\newcommand{\dbZ}{\mathbb{Z}}
\newcommand{\dbQ}{\mathbb{Q}}
\newcommand{\dbK}{\mathbb{K}}
\newcommand{\dbC}{\mathbb{C}}

\newcommand{\dbF}{\mathbb{F}}
\newcommand{\dbR}{\mathbb{R}}

\newcommand{\calO}{{\mathcal O}}

\newcommand{\calF}{{\mathcal F}}

\newcommand{\calK}{{\mathcal K}}

\newcommand{\dbH}{{\mathbb H}}

\newcommand{\philbert}{PSL_2(\mathcal{O}_k)}
\newcommand{\hilbert}{SL_2(\mathcal{O}_k)}
\newcommand{\pchain}{$p$-chain spectral sequence }
\newcommand{\orf}[1]{Or(#1,\calF)}

\newcommand{\orfin}[1]{Or(#1,\fin)}
\newcommand{\pt}[1]{*_{#1}}
\newcommand{\hompt}[3]{H_n^{Or(#1)}(\pt{#2},\pt{#3};\dbK)}
\newcommand{\assembly}[2]{H_n^{Or(#1)}(\pt{#2};\dbK)}

\newcommand{\vcyc}{V\text{\tiny{\textit{CYC}}}}
\newcommand{\fin}{F\text{\tiny{\textit{IN}}}}
\newcommand{\All}{A\text{\tiny{\textit{LL}}}}
\newcommand{\func}[3]{#1:#2\to#3}

\usepackage{anysize}
\newcommand{\nbeq}{\begin{equation}}
\newcommand{\neeq}{\end{equation}}
\newcommand{\beq}{\begin{equation*}}
\newcommand{\eeq}{\end{equation*}}
\title[$K$-theory of Hilbert modular group]{On the algebraic $K$-theory of the Hilbert modular group}
\author{Mauricio Bustamante}
\givenname{Mauricio}
\surname{Bustamante}
\address{Department of Mathematical Sciences\\
Binghamton University\\\newline
4400 Vestal Pkwy E\\ Binghamton, NY 13902\\USA}
\email{bustamante@math.binghamton.edu}
\author{Luis Jorge S\'anchez Salda\~na}
\givenname{Luis Jorge}
\surname{S\'anchez Salda\~na}
\address{Centro de Ciencias Matem\'aticas\\
UNAM, Campus Morelia \\\newline
Morelia \\Michoac\'an C.P. 58190\\M\'exico}
\email{luisjorge@matmor.unam.mx}
\begin{document}
\begin{abstract}
We give formulas for the Whitehead groups and the rational $K$-theory groups
of the (integer group ring of the) Hilbert modular group in terms of its
maximal finite subgroups.
\end{abstract}
\maketitle
%%%%%%%%%%%%%%%%%%%%%%%%%%%%%%%%%%%%%%%%%%%%%%%%%%%%%%%%%%%%%%%%%%%%%%
%%%%%%%%%%%%%%%%%%%%%%%%%%%%%%%%%%%%%%%%%%%%%%%%%%%%%%%%%%%%%%%%%%%%%%
%%%%%%%%%%%%%%%%%%%%%%%%%%%%%%%%%%%%%%%%%%%%%%%%%%%%%%%%%%%%%%%%%%%%%%
\section{Introduction}
The algebraic $K$-theory of the integral group ring of a discrete group
$G$ is known to encode topological invariants of (high dimensional) manifolds $X$ or spaces whose
fundamental group $\pi_1X$ is isomorphic to $G$. Examples come in many guises:
the obstructions to a finitely dominated space having the homotopy type of 
finite $CW$-complex and to an open smooth manifold being the interior of a compact
smooth manifold with boundary are elements of $\widetilde{K}_0(\dbZ G)$ 
(or $Wh_0(G)$ in our notation below). The obstruction to an $h$-cobordism admitting
a product structure $X\times I$ is an element in the Whitehead group $Wh(G)$ 
(certain quotient of $K_1(\dbZ G)$ denoted in this article by $Wh_1(G)$) 
and the uniqueness up to isotopy of these product structures on a compact $h$-cobordism 
has to do with a quotient of $K_2(\dbZ G)$ named $Wh_2(G)$.
In fact an entire sequence of groups $Wh_n(G)$ arising from the higher algebraic
$K$-theory of a group ring can be defined to contain the
type of invariants needed to deal with parametrized phenomena in topology, e.g.
pseudoisotopy theory. These groups $Wh_n(G)$ associated to a group $G$ are called 
\textit{Whithead groups of} $G$ and appeared defined (for larger $n$) 
by Waldhausen in \cite{W78}.
Thus, explicit calculations
of algebraic $K$-theory groups and their associated Whitehead groups $Wh_n$ 
become relevant in geometric topology. 
Unfortunately such calculation is generally a difficult
task to complete, for instance the algebraic $K$-theory of $\dbZ$ is not 
known yet in full generality. However, sometimes it is possible to approach the lower
$K$-theory and the higher $K$-theory modulo torsion via certain generalized
homology theory. One instance where this can be done is the integral
group ring of the Hilbert modular group $\philbert$. In fact, we prove:

\begin{theorem}\label{philbert}
Let $k$ be a totally real number field of finite degree and let $\mathcal{O}_k$ be its ring of algebraic integers. Then for all $q$,
$$Wh_q(\philbert)\cong \bigoplus_{(M)} Wh_q(M)$$
where the sum runs over the conjugacy classes of maximal finite subgroups of $\philbert$.
\end{theorem}
As for the non-projective Hilbert modular group $\hilbert$, we find that
its Whitehead group is determined by $K_1(\dbZ[\philbert])$. In fact, we have:
\begin{theorem}\label{hilbert}
Let $k$ be a totally real number field of finite degree and let $\mathcal{O}_k$ be its ring of algebraic integers. Then 
$$Wh_q(\hilbert)\cong K_q(\dbZ [\philbert])$$
for $q\leq 1$.
\end{theorem}
Although the higher $K$-theory groups of $\philbert$ seem to be harder
to compute, a lot more can be said about its higher rational $K$-theory. For a finite
group $F$, let
\begin{itemize}
\item $r(F)$ denote the number distinct real irreducible representations of $F$, 
\item $c(F)$ be the number of those real representations that are of complex type,
\item $q(F)$ be the number of distinct rational irreducible representations of $F$,
\item $k_p(F)$ be the number of irreducible representations of $F$ over
the $p$-adic numbers $\dbQ_p$ and
\item $r_p(F)$ be the number of irreducible representations of $F$ over the field with $p$ elements $\dbF_p$.
\end{itemize}
Denote the rank of an abelian group by $\text{rk}$.
Then we have the following
\begin{theorem}\label{rational}
Let $k$ be a totally real number field of finite degree and let $\mathcal{O}_k$ be its ring of algebraic integers. If $G=\philbert$, then 
\beq
\text{rk}\left(K_q(\dbZ[G])\right)-\text{rk}\left(H_q(BG;\mathcal{K}(\dbZ))\right)=
\begin{cases}
\displaystyle\sum_{(M)}r(M)-m &\hspace{-3.5cm}\text{ if } q>2\text{ and }q\equiv 1\text{ mod }4,\\
\displaystyle\sum_{(M)}c(M)&\hspace{-3.5cm}\text{ if } q>2\text{ and } q\equiv 3\text{ mod }4,\\
\displaystyle\sum_{(M)}(r(M)-q(M))                &\hspace{-0.5cm}\text{ if } q=1,\\
%0	  					   &\hspace{-1.5cm}\text{ if } q=0\\
m-\displaystyle\sum_{(M)}\left[\displaystyle\sum_{p||M|}q(M)-(k_p(M)-r_p(M))\right]&\hspace{-0.5cm}\text{ if } q=-1,\\
0                &\hspace{-0.5cm}\text{ } \text{otherwise}
\end{cases}
\eeq
\noindent where $m$ is the number of conjugacy classes of maximal finite subgroups in $G$, 
the sums range over the conjugacy classes of maximal finite
subgroups of  $G=\philbert$ and $H_q(BG;\calK(\dbZ))$ denotes the $q$-th homology of $BG$ with coefficients in the $K$-theory spectrum $\mathcal{K}(\dbZ)$.
\end{theorem}
The proof of Theorem \ref{rational}, together with Theorem \ref{hilbert}
lead us to expressions for the
the classical Whitehead group of the non-projective Hilbert modular group and
the lower reduced $K$-theory of $\hilbert$.
\begin{corollary}\label{WhSL}
\beq
Wh_1(\hilbert)\simeq Wh_1(\philbert)\oplus\philbert^{ab}\oplus\Z_2 ,
\eeq
where $\philbert^{ab}$ is the abelianization of $\philbert$.
Also
\beq
Wh_0(\hilbert)\simeq Wh_0(\philbert)\oplus\dbZ ;
\eeq
and
\beq
Wh_{-1}(\hilbert)\simeq\bigoplus_{(M)}K_{-1}(\dbZ[M]).
\eeq
where the sum are taken over the conjugacy classes of maximal finite subgroups of $\philbert$.
\end{corollary}

The main tool we use to prove our main results is the $K$-theoretic 
\textit{Farrell-Jones Isomorphism conjecture for a group}. When this
conjecture is verified for a group $\Gamma$, then one can potentially compute
the algebraic $K$-theory of the group ring $\dbZ[\Gamma]$ by first 
determining the algebraic $K$-theory of its virtually cyclic subgroups and
the structure of the restricted orbit category
 $Or(\Gamma,\vcyc)$. This will be explained in more detail
in Section \ref{preliminaries}. Then a spectral sequence argument can give
rise to the calculation of some of the $K$-theory groups. This method 
has proved to be effective in several cases, for example $2$- and $3$-dimensional crystallographic groups \cite{pearson98}, \cite{AO06}, 
cocompact Fuchsian groups \cite{BJPP01}, \cite{BJPP02}, \cite{LS00}, 
Bianchi groups \cite{bianchi}, braid groups  \cite{JPM06}, \cite{JPM10}, hyperbolic reflexion groups  \cite{LO09}, virtually free groups \cite{JPLMP11}, \cite{JPS14}, etc.

The spectral sequence we use in our calculation is the $p$-chain spectral
sequence of Davis and L\"uck which we review in Section \ref{pchain}. This
spectral sequence turned out to be convenient due to the  structure of the finite subgroups of the Hilbert modular group. We go over this in Section \ref{HMG}.
Section \ref{proofs} is devoted to the proof of Theorems \ref{philbert}, 
\ref{hilbert} and \ref{rational}.

%%%%%%%%%%%%%%%%%%%%%%%%%%%%%%%%%%%%%%%%%%%%%%%%%%%%%%%%%%%%%%%%%%%%%%%%%%%%%%%%%%%%%%%%%%%%%%%%%%%%%%%%%%%%%%%%%%%%%%%%%%%%%%%%%%%%%%%%%%%%%%%%%%%%%%%%%%%%
\section*{Acknowledgements}
We are grateful to Tom Farrell for suggesting this problem to us 
and D. Juan-Pineda for helpful comments on a preliminary version of this paper. The first author acknowledges the second author for teaching him how to
navigate through the somewhat dense literature that surrounds this topic. The second author wishes to thank the Binghamton University for its hospitality during the time where part of this paper was developed. The first author was supported by a Binghamton University Dissertation Fellowship. The second author was supported by PAPIIT-UNAM grant 105614, CONACYT Research grant 151338 and CONACYT doctoral scholarship.
%%%%%%%%%%%%%%%%%%%%%%%%%%%%%%%%%%%%%%%%%%%%%%%%%%%%%%%%%%%%%%%%%%%%%%
%%%%%%%%%%%%%%%%%%%%%%%%%%%%%%%%%%%%%%%%%%%%%%%%%%%%%%%%%%%%%%%%%%%%%%
%%%%%%%%%%%%%%%%%%%%%%%%%%%%%%%%%%%%%%%%%%%%%%%%%%%%%%%%%%%%%%%%%%%%%%
\section{Preliminaries}\label{preliminaries}
\par The essential tool in our calculation is the Farrell-Jones Isomorphism
conjecture. To state this conjecture in a convenient way, we recall some
definitions from \cite{DL98}.

From now on we only consider discrete groups. 
A family of subgroups $\calF$ of a group $G$ is always assumed to be closed under conjugation and under taking subgroups. We are specially interested in the
following families of subgroups:
\begin{itemize}
	\item $\All$ of all subgroups of $G$;
	\item $\vcyc$ of all virtually cyclic subgroups of $G$, i.e. subgroups which have a possibly finite cyclic subgroup of finite index;
	\item $FbC$ of finite subgroups and all virtually cyclic subgroups of the form $F\rtimes \dbZ$ with $F$ finite;
	\item $\fin$ of all finite subgroups;
	\item $Tr$ consisting of the trivial subgroup.
\end{itemize}

Note that each family is contained in the previous one.

\begin{definition}
	Let $G$ be a group and $\calF$ be a family of subgroups. The restricted orbit category $\orf{G}$ is the category whose objects are homogenous spaces,
also called orbits, $G/H$, $H\in \calF$ and whose morphisms are $G$-maps.\\
	Let $Or(G)$ denote $Or(G,\All)$, and the set of $G$ maps between the orbits $G/H$ and $G/K$ is denoted by $mor_G(G/H, G/K)$.
\end{definition} 
Note that every element in $mor_G(G/H,G/K)$ is of the form $\func{R_a}{G/H}{G/K}$ $gH \mapsto ga^{-1}K$, provided $aHa^{-1} \subseteq K$. And that $R_a=R_b$ if and only if $ab^{-1} \in K$.\\

\begin{definition}\label{orspaces}
Let $G$ be a group and $\calF$ a family of subgroups of $G$. 
	\begin{itemize}
    \item An $\orf{G}$-space (resp. $\orf{G}$-spectrum) is a functor $\orf{G}\to SPACES$ (resp. $\orf{G}\to SPECTRA$), where 
SPACES (resp. SPECTRA) is the cateogry of compactly generated topological spaces (resp. spectra). See \cite[p. 207]{DL98}.	
	\item We denote by $\pt{\calF}$ the $Or(G)$-space defined as
\beq
 \pt{\calF}(G/H) =
  \begin{cases}
   \text{point} & \text{if }  H\in\calF  \\
   \emptyset       &\text{   }otherwise.
  \end{cases}
\eeq
\end{itemize}

%		\item A covariant (resp. contravariant) $\orf{G}$-space $X$ is a covariant (resp. contravariant) functor $\func{X}{\orf{G}}{SPACES}$, where %SPACES is the cateogry of compactly generated topological spaces.
%		\item A convariant (resp. contravariant) $\orf{G}$-spectrum $\textbf{E}$ is a covariant (resp. contravariant) functor %$\func{\textbf{E}}{\orf{G}}{SPECTRA}$, where SPECTRA is the category of spectra and strong maps of spectra, see \cite[p. 207]{DL98}.
%		\item A function $\func{f}{X}{Y}$ between $\orf{G}$-spaces of the same variance is a natural transformation of functors. We can also define %functions between $\orf{G}$-spectra in a similar way, just keeping track of the structural maps, see \cite[pp. 207-208]{DL98}.
%		\item Given a contravariant $\orf{G}$-space $X$ and a covariant $\orf{G}$-space $Y$, we define the space $X\otimes_{\orf{G}} Y$ as
%		$$X\otimes_{\orf{G}} Y = \bigsqcup_{H\in \calF} X(G/H)\times Y(G/H) \diagup_\sim $$
%		where the equivalence relation is given by $(X(R_a)(x),y) \sim (x,Y(R_a)(y))$ for all $R_a:G/H \to G/K$, $gH\mapsto ga^{-1}K$, and for all %$x\in X(G/K),\ y \in Y(G/H)$.\\
%		One can define in a similar fashion the spectrum $X\otimes_{\orf{G}} \textbf{E}$ for $X$ a contravariant $\orf{G}-space$ and $\textbf{E}$ a %covariant $\orf{G}$-spectrum.
%		\item We denote by $\pt{\calF}$ the $Or(G)$-space defined as
%\beq
% \pt{\calF}(G/H) =
%  \begin{cases}
%   \text{point} & \text{if }  H\in\calF  \\
%   \emptyset       &\text{   }otherwise.
%  \end{cases}
%\eeq
%\end{itemize}
\end{definition}
Given a $G-\calF$-space  $Z$, i.e. a $G$-space $Z$ all of whose isotropy
groups belong to $\calF$, we define the fixed point (contravariant) $Or(G,\calF)$-space $\func{map_G(\_,Z)}{Or(G,\calF)}{SPACES}$ by
$G/H\mapsto map_G(G/H,Z)=Z^H$. In particular, if $Z$ is a $G-\calF$-CW complex we say that $map_G(\_,Z)$ is a free $Or(G,\calF)$-CW complex.
\begin{remark}
An unreduced homology theory for contravariant $\orf{G}$-spaces with coefficientes in a covariant $\orf{G}$-spectrum can be defined. This homology theory, denoted by $H_*^{\orf{G}}(X,Y;\textbf{E})$,
is constructed in \cite{DL98} in order to establish the Farrell-Jones conjecture.
It satisfies the Weak Homotopy Equivalence (WHE) axiom, i.e., given a weak homotopy equivalence of pairs of contravariant $\orf{G}$-spaces $\func{(f,g)}{(X,A)}{(Y,B)}$ then the homology groups $H_p^{\orf{G}}(X,A;\textbf{E})$ and $H_p^{\orf{G}}(Y,B;\textbf{E})$ are isomorphic for all $p\in \dbZ$.
\end{remark}
It can be shown that $\assembly{G}{\calF}\cong \pi_n( \displaystyle hocolim_{\orf{G}}\ \dbK )$, where $\dbK$ denotes the $K$-theory spectrum defined in \cite[Section 2]{DL98}. 
Also 
$$\assembly{G}{\All}\cong \pi_n( \displaystyle hocolim_{Or(G)}\ \dbK )\cong K_n(\dbZ [G])$$
since $Or(G)$ has a final object $G/G$ and 
$\pi_n(\dbK (G/G))\cong K_n(\dbZ [G])$.
It is worth mentioning that in \cite[Section 2]{DL98} an $\orf{G}$-spectrum
$\dbK$ is defined for every
group $G$, in such a way that $\pi_i(\dbK(G/H))=K_n(\dbZ H)$ for all $n\in \dbZ$.

In their seminal paper (\cite{FJ93}) Farrell and Jones established their
famous isomorphism conjecture for the $K$-theory, 
$L$-theory and Pseudoisotopy functors. Here we consider the $K$-theoretic
version of the conjecture as stated by Davis and L\"uck in \cite{DL98}.

\begin{conjecture}[The Farrell-Jones isomorphism conjecture]For any group $G$ the following assembly map, induced by inclusion of $Or(G)$-spaces, is an isomorphism

\nbeq\label{FJ}\tag{$\ast$}
\func{A_{\vcyc,\All}}{H^{Or(G)}_n(*_{\vcyc};\dbK)}{H^{Or(G)}_n(*_{\All};\dbK)\cong K_n(\dbZ G)}.
\neeq
\end{conjecture}
Once the Farrell-Jones conjecture has been verified for a group $G$, one can
hope to compute $K_n(\dbZ[G])$ by computing the left hand side of \eqref{FJ}.
The later is a generalized homology theory that can be approached, for 
example, via Mayer-Vietoris sequences, Atiyah-Hirzebruch-type spectral
sequences or, as in our case, the $p$-chain spectral sequence.

The Whitehead groups $Wh_n(G)$ of $G$ appear in this context as follows
\begin{proposition}\label{waldhausen}
\cite[Prop. 15.7]{W78} Let $G$ be a group. Then 
$Wh_n(G)\cong H_n^{Or(G)}(*_{\All},*_{Tr};\dbK)$ for all $n\in\dbZ$.
In fact they fit in a long exact sequence

$$\cdots\to H_n(BG;\dbK(G/1)) \to K_n(\dbZ G)\to Wh_n(G) \to
H_{n-1}(BG;\dbK(G/1))\to \cdots$$
where $H_n(BG;\dbK(G/1))$ is the classical generalized homology theory with coefficients in the spectrum $\dbK(G/1)$ which has as homotopy groups the algebraic \ktheory of the group ring $\dbZ G$.
\end{proposition}
\par When the Farrell-Jones conjecture holds for a group $G$ we obtain:
\begin{corollary}\label{fjwhitehead}
	Let $G$ be a group. Suppose that the Farrell-Jones conjecture holds for $G$, then $Wh_n(G)$ is isomorphic to $H_n^{Or(G)}(\pt{\vcyc},\pt{Tr};\dbK)$
for all $n\in\dbZ$.
\end{corollary}
\par One usually wonders
whether a smaller family of subgroups than $\vcyc$ would suffice to succeed
in the calculation of the Whitehead groups of a group or the $K$-theory of a group ring. The next two theorems tell us something about it.
\begin{theorem}\cite[A.10]{FJ93}\label{reducefin}
	Let $G$ be a group. Suppose that for any finite by cyclic subgroup $V$ of $G$, the assembly map
$$\func{A_{\fin,\All}}{H_n^{Or(V)}(\pt{\fin};\dbK)}{H_n^{Or(V)}(\pt{\All};\dbK)\cong K_n(\dbZ V)}$$
	is an isomorphism. Then $\assembly{G}{\fin}\cong \assembly{G}{FbC}$.
\end{theorem}	
\begin{remark}
The obstructions to $A_{\fin,\All}$ being an isomorphism is the so-called
Nil groups. The vanishing of these groups allows one to consider only
the family of finite subgroups. 
Thus under the hypothesis of Theorem \ref{reducefin} and using results
of \cite{DKR11}, \cite{DQR11}, one can actually compute the Whitehead groups
of a group $G$ by considering only its family of finite groups. In other words, if $G$ satisfies the Farrell-Jones conjecture then
	\begin{align*} 
	Wh_n(G) &\cong \hompt{G}{\All}{Tr} \\ 
			&\cong \hompt{G}{\vcyc}{Tr}\\
			&\cong \hompt{G}{FbC}{Tr}\\
			&\cong \hompt{G}{\fin}{Tr}.
	\end{align*}
\end{remark}
\par Here is one more way in which one can replace the family of subgroups
of a group by a smaller family.
\begin{theorem}\cite[Corollary 3.9]{DL03}\label{reducemax}
	Let $\calF$ be a family of subgroups of $G$. Denote by $M\calF \subset \calF$ the subfamily consisting of: all maximal elements in $\calF$, the groups in $\calF$ which are contained in more than one maximal element and the groups in $\calF$ which are contained in no maximal element of $\calF$. Then for all $n$
	$$\assembly{G}{\calF}\cong \assembly{G}{M\calF}.$$
\end{theorem}
%\begin{enumerate}
%	\item As we want to calculate $Wh(G)$, first we use Theorem $\ref{waldhausen}$ to translate the problem into compute $\hompt{G}{All}{Tr}$.
%	\item Once translated the problem we want to reduce the family of subgroups that we are using. Here is where we use the Farrell-Jones conjecture, to reduce the family from $All$ to $Vcyc$, so that now we want to compute $\hompt{G}{Vcyc}{All}$.
%	\item It turns out that we can always reduce the family from $Vcyc$ to $Fbc$ due to (\cite{DKR11}, \cite{DQR11}).
%	\item And finally we would like to reduce from $Fbc$ to $Fin$. But this cannot be always done, the reason are the so-called Nil-groups. Even though we have Theorem \ref{reducefin} which give us hypothesis to achieve our reduction.
%	\item Given any family $\calF$ we can reduce the family $\calF$ to the family consisting of the maximal groups of $\calF$ plus the groups in $\calF$ which are contained in more than one maximal element plus all the groups in $\calF$ which are not contained in any maximal group of $\calF$. In particular, we can apply this to the family of finite subgroups of $G$. The formal statement is in Theorem \ref{reducemax}.
%	\item When the reduction of family have been performed, then we can use a spectral sequence to compute the corresponding homology group. In our case we will use the \pchain.
%\end{enumerate}

%%%%%%%%%%%%%%%%%%%%%%%%%%%%%%%%%%%%%%%%%%%%%%%%%%%%%%%%%%%%%%%%%%
%%%%%%%%%%%%%%%%%%%%%%%%%%%%%%%%%%%%%%%%%%%%%%%%%%%%%%%%%%%%%%%%%%
%%%%%%%%%%%%%%%%%%%%%%%%%%%%%%%%%%%%%%%%%%%%%%%%%%%%%%%%%%%%%%%%%%%%
\section{The \pchain}\label{pchain}
In this section we establish a special case of the $p$-chain spectral
sequence of Davis and L\"uck \cite{DL03}. Besides the Farrell-Jones
conjecture, this will be the main ingredient in our calculation. The
$p$-chain spectral sequence converges to $H^{\orf{G}}(X;\textbf{E})$. By
reasons that will become evident in the next section, it will be enough
to restrict ourselves to the restricted orbit category $\orfin{G}$ and $X$ the one point $Or(G,\fin)$-space. This
category has the property that every endomorphism is an isomorphism and
for each object $G/K$, $Aut(G/K)$ acts freely in $mor_G(G/H,G/K)$. 
Let $\overline{G/H}$ denote the isomorphism class of $G/H$. Note that
there is a partial order on the set of isomorphism classes defined by
$\overline{G/H}\leq \overline{G/K}$ if $mor_G(G/H,G/K)$ is nonempty. We
also write $\overline{G/H}<\overline{G/K}$ if 
$\overline{G/H} \leq \overline{G/K}$ and 
$\overline{G/K} \nleq \overline{G/H}$.
\begin{definition}
Let $G$ be a group and $\calF$ a family of subgroups.
A sequence of isomorphism classes of objects
$c:=\{\overline{G/H_0},\overline{G/H_1},\cdots,\overline{G/H_p}\}$
in $Or(G,\calF)$ is called a $p$-chain if
\beq
\overline{G/H_0}<\overline{G/H_1}<\cdots < \overline{G/H_p}.
\eeq
\end{definition}
Associated to a $p$-chain
$c=\{\overline{G/H_0},\overline{G/H_1},\cdots,\overline{G/H_p}\}$ there
are $Aut(G/H_p)$-$Aut(G/H_0)$-spaces $S(c)$ defined by

\beq
S(c) =
  \begin{cases}
   Aut(G/H_0) & \hspace{-1.5cm} \text{if } p=0 \\
   mor_G(G/H_{0},G/H_1) & \hspace{-1.5cm}\text{if } p=1\\
   mor_G(G/H_{p-1},G/H_p) \times_{Aut(G/H_{p-1})}  \cdots \times_{Aut(G/H_{1})} mor_G(G/H_{0},G/H_1) &\text{if } p\geq 2
  \end{cases}
\eeq
The next result is proven (in more generality)
in \cite[Lemma 2.11 and Remark 2.13]{DL98}
\begin{theorem}[The \pchain] 
	Let $G$ be a group and let $\mathcal{J}\subset\fin$ be any subfamily
of the family of finite subgroups of $G$. Then there is a spectral 
sequence whose first page is given by
\beq
E^1_{p,q}=\bigoplus_{\substack{p-\text{ chains } c\\ \text{ in } Or(G,\mathcal{J})}}
H_q^{Aut(G/H_0)} (pt\times_{Aut(G/H_p)} S(c); \dbK)
\eeq
which converges to $H_{p+q}^{Or(G)}(\ast_{\mathcal{J}}; \dbK)$.
\end{theorem}
\begin{remark}
This spectral sequence is very manageable when one can control the length
of the $p$-chains. In our case, when $G$ becomes the Hilbert modular group,
we will see that we can work in a family of subgroups in which 
no $p$-chains for $p\geq 2$ have to be considered in the restricted orbit category.
This will make the spectral sequence collapse quickly and explicit
calculations can be done.
\end{remark}
%%%%%%%%%%%%%%%%%%%%%%%%%%%%%%%%%%%%%%%%%%%%%%%%%%%%%%%%%%%%%%%%%%%%%%
%%%%%%%%%%%%%%%%%%%%%%%%%%%%%%%%%%%%%%%%%%%%%%%%%%%%%%%%%%%%%%%%%%%%%%
%%%%%%%%%%%%%%%%%%%%%%%%%%%%%%%%%%%%%%%%%%%%%%%%%%%%%%%%%%%%%%%%%%%%%%
%%%%%%%%%%%%%%%%%%%%%%%%%%%%%%%%%%%%%%%%%%%%%%%%%%%%%%%%%%%%%%%%%%%%%%
%%%%%%%%%%%%%%%%%%%%%%%%%%%%%%%%%%%%%%%%%%%%%%%%%%%%%%%%%%%%%%%%%%%%%%
%%%%%%%%%%%%%%%%%%%%%%%%%%%%%%%%%%%%%%%%%%%%%%%%%%%%%%%%%%%%%%%%%%%%%%

\section{The Hilbert Modular Group}\label{HMG}
In this section we review the definition and some basic properties of
the Hilbert modular group.
The results we state without proof in this section can be found
for example in \cite{freitag}. For additional information about the Hilbert modular group we refer the reader to \cite{hirzebruch}, \cite{geer}.

A totally real number field $k$ is an algebraic extension of $\dbQ$ such that all its embedings $\sigma_i:k\to \dbC$ have image contained in $\dbR$. Let $k$ denote a totally real number field of degree $n$ and
$\calO_k$ its ring of integers. The \textit{Hilbert modular group} $PSL_2(\calO_k)$ is defined to be the quotient of the special linear
group of 2 by 2 matrices $SL_2(\calO_k)$ with entries in $\calO_k$ by the subgroup consisting of $\{I,-I\}$, where $I$ denotes the identity matrix; 
in other words 
\beq
PSL_2(\calO_k)=SL_2(\calO_k)/\{I,-I\}.
\eeq

Note that if $k=\dbQ$, then $PSL_2(\calO_k)=PSL_2(\dbZ)$ is nothing but the classical
modular group. However $PSL_2(\calO_k)$ \textit{is not} a discrete subgroup
of $PSL_2(\dbR)$ if $n\geq 2$. Yet it does act properly and discontinuously on the
$n$-fold product $\dbH^n=\dbH\times\cdots\times\dbH$ of upper half planes, by fractional linear transformations in each of the $n$ factors, via the $n$ different
embeddings of $k$ into $\dbR$. Thus the Hilbert modular group is a discrete subgroup
of $PSL_2(\dbR)^n=PSL_2(\dbR)\times\cdots\times PSL_2(\dbR)$.
For example, if we let $d$ be a square-free positive integer and $k=\dbQ(\sqrt{d})$, 
then there are two embeddings of $k$ into $\dbR$, namely
\beq
\sigma_1: s+t\sqrt{d}\mapsto s+t\sqrt{d}\ \ \ \ \ \ \text{and}\ \ \ \ \ \ \sigma_2:s+t\sqrt{d}\mapsto s-t\sqrt{d},
\eeq
for $s,t\in\dbQ$. In this case, a matrix  
$\begin{pmatrix}
\alpha & \beta\\
\gamma & \delta
\end{pmatrix}\in PSL_2(O_k)$ acts on $(z_1,z_2)\in\dbH\times\dbH$ by 
\beq
\begin{pmatrix}
\alpha & \beta\\
\gamma & \delta
\end{pmatrix}
(z_1,z_2)=\left(\frac{\sigma_1(\alpha)z_1+\sigma_1(\beta)}{\sigma_1(\gamma)z_1+\sigma_1(\delta)},
\frac{\sigma_2(\alpha)z_2+\sigma_2(\beta)}{\sigma_2(\gamma)z_2+\sigma_2(\delta)}\right).
\eeq
\par The action of $PSL_2(\calO_k)$ on $\dbH^n$ is not free. One can detect
points with non-trivial isotropy by the following lemma (Compare with the analogous situation of 
$PSL_2(\dbZ)$ acting on $\dbH$ by M\"obius transformations):
\begin{lemma}
Let $h\in PSL_2(\calO_k)$ and dentote by 
$\sigma_i:PSL_2(\calO_k)\hookrightarrow PSL_2(\dbR)$, $i=1,\ldots, n$ the 
canonical embeddings of $PSL_2(\calO_k)$ into $PSL_2(\dbR)$. Then the following
conditions are equivalent:
\begin{enumerate}
\item $\sigma_1(h),\ldots,\sigma_n(h)\in PSL_2(\dbR)$ are elliptic matrices,
i.e. their traces satisfy 
\beq
[Tr\: \sigma_i(h)]^2-4<0.
\eeq
\item $h$ has finite order.
\item $h$ has a unique fixed point.
\end{enumerate}
\par Moreover the stabilizer $\Gamma_z$ of any point $z\in\dbH^n$ in 
$PSL_2(\calO_k)$ is a finite cyclic group.
\end{lemma}

For example the matrix $\begin{pmatrix}
\sqrt{2} & -1\\
1 &0
\end{pmatrix}$
is an elliptic element of $PSL_2\left(\calO_{\dbQ(\sqrt{2})}\right)$ of order
$4$ that fixes the point $\frac{\sqrt{2}}{2}((1+i),(-1+i))\in\dbH\times\dbH$.
\begin{remark}
$\dbH^n$ turns out to be a classifying space for proper actions of 
the Hilbert modular group $PSL_2(\calO_k)$, where the fixed point sets by
finite subgroups are not only contractible but consist of a single point. This
property will simplify the calculation of the $K$-theory groups enormously.
\end{remark}
Perhaps more importantly for our purposes, $PSL_2(\calO_k)$ satisfies the
conditions M, NM and FJ specified in \cite[Section 4]{DL03}:
\begin{lemma}\label{MNM}
The Hilbert modular group $PSL_2(\calO_k)$ satisfies the following three
properties:
\begin{itemize}
\item[(M)] Every finite subgroup of $PSL_2(\calO_k)$ is contained in a unique maximal
finite subgroup.
\item[(NM)] If $M$ is a maximal finite subgroup of $PSL_2(\calO_k)$ then
$N(M)=M$, where $N(M)$ denotes the normalizer of $M$ in $PSL_2(\calO_k)$.
\item[(FJ)] The Isomorphism Conjecture of Farrell and Jones for algebraic $K$-theory
is true for $PSL_2(\calO_k)$.
\end{itemize}
\end{lemma}
\begin{proof}
Note that by the lemma above if $H$ is a finite subgroup of 
$PSL_2(\calO_k)$ then $H$ is contained in the finite
stabilizer $\Gamma_z$ of a point $z\in\dbH^n$. $\Gamma_z$ is a finite maximal subgroup
of $PSL_2(\calO_k)$. For if $F$ is a finite subgroup of
 $PSL_2(\calO_k)$ containing $\Gamma_z$, then there must exist $y\in\dbH^n$
such that $f\cdot y=y$ for all $f\in F$. In particular, every element of $\Gamma_z$
would fix $y$ and by uniqueness of the fixed points, $y=z$, that is to say $f\in\Gamma_z$. 
This proves that $PSL_2(\calO_k)$ satisfies property (M).
\par Let $M$ be a maximal finite subgroup of $PSL_2(\calO_k)$. Then by the
previous paragraph, $M=\Gamma_z$,
where $\Gamma_z$ is the stabilizer of some point $z\in\dbH^n$. Let now $g\in N(M)$,
then $gfg^{-1}\in M$ for some all $f\in M=\Gamma_z$. Hence $gfg^{-1}z=z$ which
implies that $fg^{-1}z=g^{-1}z$, and by the uniqueness of the fixed points
$g^{-1}z=z$, i.e. $g\in M$. This proves that $PSL_2(\calO_k)$ satisfies
the property (NM).
\par As for property (FJ) we have to note that $\philbert$ is a lattice
in the Lie group $PSL_2(\dbR)^n$. The Farrell-Jones conjecture has been
proven for this type of groups in \cite{KLR14}.
\end{proof}
\begin{remark}
Conditions (M) and (NM) can be interpreted in a somewhat geometric way
as follows: Let $\mathcal{J}$ be the family of groups $G$ satisfying (FJ) and
for which there is a model for $\underline{E}G$ with the property that 
every fixed point set by a finite subgroup of $G$ consists of a single point.
Then it is clear that if $G\in\mathcal{J}$, then $G$ satisfies (M) and (NM)
and (FJ). It is worth noticing that such family $\mathcal{J}$ is closed
under free products.
\end{remark}
\par If a group $G$ has properties (M) and (NM) its only possible
virtually cyclic subgroups are very limited (Compare \cite[p.100]{DL03}). 
\begin{lemma}\label{vcyc}
Let $G$ be a group having properties (M) and (NM), then every 
infinite virtually
cyclic subgroup of $G$ is isomorphic to either $\dbZ$ or $\dbZ_2\ast\dbZ_2$.
\end{lemma}
\begin{proof}
Every infinite virtually cyclic subgroup $V$ of $G$ fits into an extension
\beq
1\to F\to V\to \Gamma\to 1
\eeq
where $\Gamma$ is either $\dbZ$ or $\dbZ_2\ast\dbZ_2$ and $F$ is finite. Since $F$ is normal in $V$ then $V\subset N(F)$ where $N(F)$ is the normalizer of $F$ in $G$. Suppose that $F$ is non-trivial and let $M$ be
the unique maximal finite subgroup of $G$ containing $F$.
It is clear that $F\subset gMg^{-1}$ if $g\in N(F)$. Hence, by uniqueness,
$g\in M$. This shows that $N(F)\subset M$ which is finite, contradicting that $V\subset N(F)$. Therefore $F$ is trivial and the conclusion follows.
\end{proof}
%%%%%%%%%%%%%%%%%%%%%%%%%%%%%%%%%%%%%%%%%%%%%%%%%%%%%%%%%%%%%%%%%%%%%%%55
%%%%%%%%%%%%%%%%%%%%%%%%%%%%%%%%%%%%%%%%%%%%%%%%%%%%%%%%%%%%%%%%%%%%%%%%%
%%%%%%%%%%%%%%%%%%%%%%%%%%%%%%%%%%%%%%%%%%%%%%%%%%%%%%%%%%%%%%%%%%%%%%%%%%
%%%%%%%%%%%%%%%%%%%%%%%%%%%%%%%%%%%%%%%%%%%%%%%%%%%%%%%%%%%%%%%%%%%%%%%%

\section{Whitehead groups and rational $K$-theory of the Hilbert modular group}\label{proofs}
In this section we prove Theorems \ref{philbert}, \ref{hilbert} and
\ref{rational}. 
\begin{proof}[\textbf{Proof of Theorem \ref{philbert}}]
Let $G=\philbert$. Recall that 
$Wh_q(G)\simeq H_q^{Or(G)}(\ast_{\All},\ast_{Tr};\dbK)$.
The first step in our calculation is to verify the
Farrell-Jones conjecture for $G$. As mentioned in Lemma \ref{MNM}, this 
follows from \cite{KLR14}. The next step is trying to reduce the family
of subgroups as much as possible. In fact, by Lemma \ref{vcyc} and the
fact that the Nil groups of the integer group rings of 
$\dbZ$ and $\dbZ_2\ast \dbZ_2$ vanish \cite[Lemma 2.5]{LS00}, we have that the assembly
map 
$$\func{A_{\fin,\All}}{H_q^{Or(V)}(\pt{\fin};\dbK)}{H_q^{Or(V)}(\pt{\All};\dbK)}$$
is an isomorphism. Hence $Wh_q(G)$ can be computed using only the family
of finite subgroups of $G$. That is to say
\beq
Wh_q(G)\simeq H_q^{Or(G)}(\ast_{\fin},\ast_{Tr};\dbK)
\eeq
Furthermore, by combining Lemma \ref{MNM} (property M) and 
Theorem \ref{reducemax}, we see that the family of subgroups considered can
be reduced further to the subfamily $M\fin$, where $M\fin$ denotes the family of maximal finite subgroups of $G$ union the trivial subgroup, as defined in Theorem \ref{reducemax}. Hence
\beq
Wh_q(G)\cong H_q^{Or(G)}(\ast_{M\fin},\ast_{Tr};\mathbb{K}).
\eeq
\par Now that we have replaced $\vcyc$ by  $M\fin$, we analyze
the structure of the $p$-chains in $Or(G,M\fin)$.
\begin{description}
\item[$0$-chains.] These are chains of the form
$\{\overline{G/H}\}$, where $H\in M\fin$. 
\item[$1$-chains.] Here we only have chains of the form $\{\overline{G/1},\overline{G/H}\}$. Note
that there are no $1$-chains of the form $\{\overline{G/H},\overline{G/K}\}$, for $H,K\in M\fin$, $H\neq 1$, 
because every morphism $G/H\to G/K$ has to be an isomorphism.
\item[$p$-chains, $p\geq 2$.] There are none because every morphism
$G/H\to G/K$, $H\neq 1$,  is an isomorphism.
\end{description}
Also notice that by considering the pair
$(\ast_{M\fin},\ast_{Tr})$ we are neglecting terms coming
from the $0$-chain $\{G/1\}$ in the $E^1$ page of the $p$-chain spectral
sequence for pairs. Particularly, no $1$-chains have to be considered in the
calculation of the Whitehead groups.
Thus the spectral sequence will only have contributions
from isomorphism classes of orbits of the form $G/H$ with $H\in M\fin$, $H\neq 1$. Now recall that $Aut(G/H)=N(H)/H=\{e\}$, where $N(H)$ is the normalizer of $H$ in $G$ and
the last equality follows from Lemma \ref{MNM} (property NM). Also
there is a correspondence between $0$-chains in $Or(G,M\fin)$ and conjugacy
classes of maximal finite subgroups in $G$.
Thus we end up
with the following $E^1$-term.
\begin{figure}[!h]
\vspace*{0cm}
    \begin{center}
    \includegraphics[scale=1]{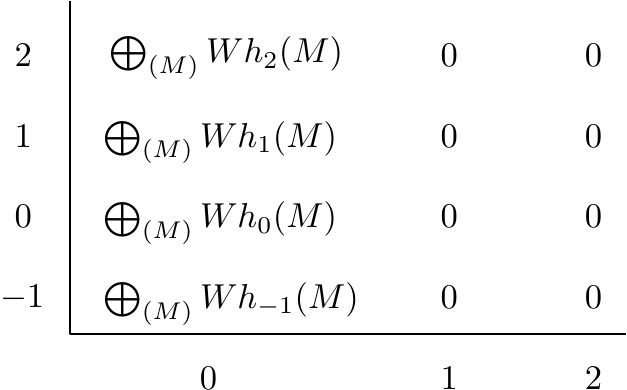}
    %\caption*{}
    \end{center}
 \end{figure}
Therefore the spectral sequence collapses at $E^1$ and the result follows. 
\end{proof}
%%%%%%%%%%%%%%%%%%%%%%%%%%%%%%%%%%%%%%%%%%%%%%%%%%%%%%%%%%%%
%%%%%%%%%%%%%%%%%%%%%%%%%%%%%%%%%%%%%%%%%%%%%%%%%%%%%%%%%%%%
\begin{proof}[\textbf{Proof of Theorem \ref{hilbert}}]
We follow the same strategy as in the proof of Theorem \ref{philbert}.
Let $G=\hilbert$. It satisfies the Farrell-Jones conjecture for the same
reasons that $\philbert$ does. Thus we can guarantee that
\beq
H_q^{Or(G)}(\ast_{Vcyc};\mathbb{K})\cong K_q(\mathbb{Z}[G])
\eeq
for all $q \in \mathbb{Z}$. The idea now is trying to reduce the family
of subgroups in order to compute the left side of this equation.
Recall that by  \cite{DKR11}, \cite{DQR11} it is always posible to 
reduce the family from $\vcyc$ to $FbC$ (see also Section \ref{preliminaries}).
Hence if we want to replace the family $\vcyc$ with $\fin$ it is enough
to prove that for any subgroup of $G$ of the form $F\rtimes\mathbb{Z}$ with $F$ finite, the assembly map $f:H_q^{Or(F\rtimes\mathbb{Z})}(\ast_{\fin};\mathbb{K})\to K_q(\mathbb{Z}[F\rtimes\mathbb{Z}])$ is an isomorphism.
To see this, note that $F\times\mathbb{Z}<F\rtimes\mathbb{Z}$ and 
the only non-trivial finite-order element that commutes with 
	an element of infinite order is 
	$-I=\left( \begin{array}{cc}
	-1 & 0 \\ 
	0 & -1
	\end{array}\right) $,
then every group in the family $FbC$ is isomorphic to either $\dbZ$ or
$\dbZ_2\times\dbZ$. By \cite[Lemma 2.5]{LS00} $f$ is an isomorphism
for $q\leq 1$ when $F$ is the trivial group or $F\cong\dbZ_2$. Therefore
\beq
H_q^{Or(G)}(\ast_{\vcyc};\mathbb{K})\cong H_q^{Or(G)}(\ast_{FbC};\mathbb{K})\cong H_q^{Or(G)}(\ast_{\fin};\mathbb{K}),
\eeq
whenever $q\leq 1$.\\
We can actually work with a smaller family of subgroups. First
note that by the same reasoning as in the proof of Lemma \ref{MNM},
every finite maximal subgroup of $G$ is of the form 
$G_z$, $z\in\underline{E}G=\mathbb{H}\times\cdots\times\mathbb{H}$,
the $n$-fold product of upper half-spaces. Hence every 
finite subgroup different from $\lbrace\pm I\rbrace$ is
contained in a unique finite maximal subgroup. Let 
$M\fin$ as defined in Theorem \ref{reducemax}, note that every element in $M\fin$ is a maximal finite subgroup, or the trivial subgroup or the group $\{I,-I \}$. By Theorem
\ref{reducemax} we have
\beq
H_q^{Or(G)}(\ast_{\fin};\mathbb{K})\cong H_q^{Or(G)}(\ast_{M\fin};\mathbb{K}).
\eeq
Now, using the Five Lemma and the previous reductions of the family of subgroups, it is straightforward to show that, for $q\leq 1$
\beq
Wh_q(G)\cong H_q^{Or(G)}(\ast_{All},\ast_{Tr};\mathbb{K})\cong H_q^{Or(G)}(\ast_{ M\fin},\ast_{Tr};\mathbb{K}).
\eeq
We now analyze the $p$-chains that appear in $Or(G,M\fin)$ 
\begin{description}
\item[$0$-chains.] These are chains of the form $\{\overline{G/1}\}$, 
$\{\overline{G/\{\pm I\}}\}$ and $\{\overline{G/H}\}$, where $H$ is a finite maximal subgroup. 
\item[$1$-chains.] We have chains of the form $\{\overline{G/1},\overline{G/\{\pm I\}}\}$,
$\{\overline{G/1},\overline{G/H}\}$ and $\{\overline{G/\{\pm I\}},\overline{G/H}\}$, with $H$ a finite maximal subgroup.
% Note that there are no $1$-chains of the form $\{G/H,G,K\}$, for $H,K\in
 %M\fin$ because every endomorphism $G/H\to G/K$ has to be an isomorphism.
\item[$2$-chains.] The only $2$-chains that we can form are of the type
$\{\overline{G/1},\overline{G/\{\pm I\}},\overline{G/H}\}$, with $H$ a finite maximal subgroup.
\item[$p$-chains, $p\geq 3$.] There are none because for 
$H,\ K$ maximal every morphism $G/H\to G/K$ has to be an isomorphism
\end{description}
Also notice that by considering the pair
$(\ast_{M\fin},\ast_{Tr})$ we are neglecting terms coming
from $p$-chains whose least element is the $0$-chain $\{\overline{G/1}\}$. Particularly, no $2$-chains have to be considered.
Thus the $E^1$ term of $p$-chain spectral sequence for pairs will 
only have contributions from orbits in $Or(G, M\fin^0)$, where
$M\fin^0:=M\fin-Tr$.\\
It is clear that $Or(G, M\fin^0)$ is equivalent to 
$Or(\philbert,M\fin)$. Thus the $E^1$-term of the $p$-chain
spectral sequence (for pairs) in $Or(G, M\fin)$ is isomorphic 
to the $E^1$-term of the
$p$-chain spectral sequence in $Or(\philbert,M\fin)$. Since the former
converges to $Wh_*(G)$ and the latter
converges to $K_*(\dbZ[\philbert])$ the result follows.
\end{proof}
%%%%%%%%%%%%%%%%%%%%%%%%%%%%%%%%%%%%%%%%%%%%%%%%%%%%%%%%%%%%
%%%%%%%%%%%%%%%%%%%%%%%%%%%%%%%%%%%%%%%%%%%%%%%%%%%%%%%%%%%%
%%%%%%%%%%%%%%%%%%%%%%%%%%%%%%%%%%%%%%%%%%%%%%%%%%%%%%%%%%%%
\begin{proof}[\textbf{Proof of theorem \ref{rational}}]
Let $G=\philbert$. Recall that $M\fin$ from Theorem \ref{reducemax} is the family
of subgroups of $G$ consisting of all maximal finite subgroups together with the
trivial subgroup. In \cite[Theorem 5.6]{Gr08}, it is proven that for any group $\Gamma$ that satisfies the Farrell-Jones
conjecture
%the torsion free part of the
%algebraic $K$-theory of a group ring is detected by the %$K$-theory of the maximal
%finite subgroups. More precisely, 
\beq
H_*^{Or(\Gamma)}\left(\ast_{\fin};\dbK\right)\otimes\dbQ\cong 
H_*^{Or(\Gamma)}\left(\ast_{\vcyc};\dbK\right)\otimes\dbQ\cong K_*(\dbZ[\Gamma])\otimes\dbQ.
\eeq
Now we use Theorem \ref{reducemax} to reduce the family from $\fin$ to $M\fin$. We start
analyizing the $p$-chains that appear in $Or(G,M\fin)$. This has been
done already in the proof of Theorem \ref{philbert}. Note, though, that this time
we are not computing the homology of a pair, so we do need
to consider $0$-chains and
$1$-chains (there are no $p$-chains in this category for $p\geq 2$.) 
The first page of the $p$-chain spectral sequence is given (after some
simplifications, see \cite[Proposition 12]{BJPP01}) by
\begin{align*}
E^1_{0q}&=H_q(BG;\calK(\dbZ))\oplus\bigoplus_{(M)}K_q(\dbZ[M]))\\
E^1_{1q}&=\bigoplus_{(M)}H_q(BM;\calK(\dbZ))\\
E^1_{pq}&=0\text{  for }p\neq 0,1,
\end{align*}
where $H_q(\_ \ ;\mathcal{K}(\dbZ))$ refers to the generalized homology theory with coefficients in the Pedersen-Weibel
$K$-theory spectrum and the sum
runs over conjugacy classes of maximal finite subgroups.
Thus the $p$-chain spectral sequence in this case looks like this:
\begin{figure}[!h]
\vspace*{0cm}
    \begin{center}
    \includegraphics[scale=1]{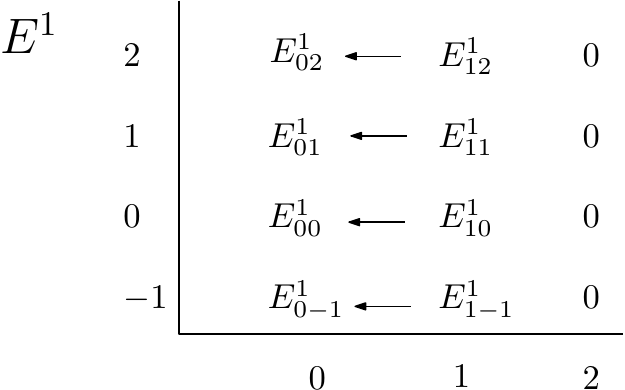}
    %\caption*{}
    \end{center}
 \end{figure}
\par The differentials $d^1:E^1_{1q}\to E^1_{0q}$ are all induced 
by two types of maps between 
$p$-chains: Those arising from $\{\overline{G/1},\overline{G/M}\}\mapsto\{\overline{G/1}\}$ and those that come
from $\{\overline{G/1},\overline{G/M}\}\mapsto\{\overline{G/M}\}$. It is shown in \cite[Lemma 3.10]{pearson98}
that in the first case the induced map in homology is
inclusion whereas the second map induces the classical assembly map
$H_q(BM;\calK(\dbZ))\to K_q(\dbZ[M])$ which in our notation is $$A_{Tr,All}:H^{Or(M)}_q(*_{Tr};\dbK)\to H^{Or(M)}_q(*_{All};\dbK).$$ In our case
this assembly map is known to be rationally injective for all $q$ 
(see \cite{BHM}). Hence, at least rationally the differentials $d^1$
are injective and the spectral sequence (rationally) collapses at $E^2$. Then
\beq
\text{rk }(K_q(\dbZ[G]))=\text{rk }(E^{\infty}_{0q})=\text{rk }(E^2_{0q})
\eeq
Note that $E^2_{0q}$ fits in an exact sequence:
\beq
0\to E^1_{1q}\otimes\dbQ\xrightarrow{d^1\otimes id_{\dbQ}}E^1_{0q}\otimes\dbQ\to 
E^2_{0q}\otimes\dbQ\to 0.
\eeq
Finally, by taking into account that every finite subgroup of the Hilbert modular
group $G$ is cyclic, an Atiyah-Hirzebruch spectral sequence calculation shows
that 
\beq
\text{rk }(H_q(BM;\mathcal{K}(\dbZ))=
\begin{cases}
1 &\text{ if }q=0\text{ or }q\equiv 1\text{ mod 4}, q>2\\
0 &\text{  } \text{otherwise}
\end{cases}
\eeq
Hence,
\begin{align*}
\text{rk }(K_q(\dbZ[G]))&=\text{rk }(E^2_{0q})\\
						  &=\text{rk }(E^1_{0q})-\text{rk }(E^1_{1q})\\
						  &=\displaystyle\sum_{(M)}\text{rk }(K_q(\dbZ[M]))
						  +\text{rk }(H_q(BG;\mathcal{K}(\dbZ))-
\displaystyle\sum_{(M)}\text{rk }(H_q(BM;\mathcal{K}(\dbZ)).
\end{align*}
Thus we have
\nbeq\label{rk}\tag{**}
\text{rk }(K_q(\dbZ[G]))-\text{rk }(H_q(BG;\mathcal{K}(\dbZ))=
\begin{cases}
\displaystyle\sum_{(M)}\text{rk }(K_q(\dbZ[M]))-m &\text{ if }q=0\text{ or }q\equiv 1\text{ mod 4}, q>2\\
\displaystyle\sum_{(M)}\text{rk }(K_q(\dbZ[M]))&\text{  }\text{otherwise},
\end{cases}
\neeq
where $m$ denotes the number of conjugacy classes of maximal finite subgroups of $G$.
By results of Bass \cite{bass}, Carter \cite{carter} and Jahren \cite{Ja09} on the
$K$-theory of finite groups, we have the following equalities:
\beq
\text{rk }(K_q(\dbZ[M]))=
\begin{cases}
r(M)&\text{ if }q>2\text{ and } q\equiv 1\text{ mod }4\\
c(M)&\text{ if }q>2\text{ and } q\equiv 3\text{ mod }4\\
r(M)-q(M)&\text{ if } q=1\\
1&\text{ if } q=0\\
1-q(M)+\displaystyle\sum_{p||M|}(k_p(M)-r_p(M))&\text{ if } q=-1\\
0   &\text{    } \text{ otherwise}.
\end{cases}
\eeq
The result then follows by substituting in the equation \eqref{rk}.
\end{proof}
\begin{proof}[Proof of Corollary \ref{WhSL}]
Recall that by Theorem \ref{hilbert} $Wh_1(\hilbert)\simeq K_1(\philbert)$.
Now notice that the differential of the spectral sequence that appears in
the proof of Theorem \ref{rational}, is injective for $q\leq 1$, see \cite[Proof of Proposition 14]{BJPP01}. This implies that 
$K_1(\philbert)=E^{\infty}_{01}=E^{2}_{01}$ and then we obtain an exact
sequence
\beq
0\to E^1_{11}\xrightarrow{d^1}E^1_{01}\to K_1(\philbert)\to 0.
\eeq
Following \cite[Proof of Proposition 14]{BJPP01} is straightforward to check  that
\beq
E^{1}_{01}=\philbert^{ab}\oplus\Z_2\oplus\bigoplus_{(M)} K_1(\dbZ[M])
\eeq
and also
\beq
E^{1}_{11}=\bigoplus_{(M)} (M\oplus\dbZ_2).
\eeq
Since $M$ is abelian, $K_1(\dbZ[M])\simeq M\oplus\dbZ_2\oplus Wh_1(M)$ and
the image of the assembly map $H_1(BM;\calK(\dbZ))\simeq M\oplus\dbZ_2\to
K_1(\dbZ[M])$ splits for each $M$. This gives us the desired
formula for the classical Whitehead group:
\beq
Wh_1(\hilbert)\simeq Wh_1(\philbert)\oplus\philbert^{ab}\oplus\Z_2.
\eeq
\par Similarly we obtain an exact sequence
\beq
0\to E^1_{10}\xrightarrow{d^1}E^1_{00}\to K_0(\philbert)\to 0,
\eeq
where
\beq
E^1_{10}=\displaystyle\bigoplus_{(M)}H_0(BM;\calK(\dbZ))\simeq
\displaystyle\bigoplus_{(M)}\dbZ
\eeq
and
\beq
E^1_{00}\simeq\dbZ\oplus\displaystyle\bigoplus_{(M)}K_0(\dbZ[M]).
\eeq
Note that 
$K_0(\dbZ[M])\simeq\dbZ\oplus Wh_0(M)$ and $Wh_0(M)$ is finite because $M$ is 
finite. The result then follows by noticing
that the assembly map 
$H_0(BM;\calK(\dbZ))\simeq\dbZ \to K_0(\dbZ[M])\simeq\dbZ\oplus Wh_0(M)$
is split injective for each $M$.
\par The same analysis at the $-1$-st level yields $Wh_{-1}(\hilbert)\simeq\bigoplus K_{-1}(\dbZ[M])$. This completes the proof of the corollary.
\end{proof}
\par We conclude this section working out a concrete example.
Let $k=\dbQ(\sqrt{5})$ the quadratic extension of the rational numbers obtained 
by adjoining $\sqrt{5}$ to $\dbQ$. In this case, $PSL_2(\calO_k)$ acts
on $\dbH\times\dbH$ via the embeddings induced by 
$\sigma_1:\sqrt{5}\mapsto\sqrt{5}$
and $\sigma_2:\sqrt{5}\mapsto -\sqrt{5}$. By Theorem \ref{philbert}, 
the Whitehead groups
of $PSL_2(\calO_k)$ are determined by its maximal finite subgroups. Hence
our problem reduces to finding all conjugacy classes of maximal finite subgroups.
By Lemma \ref{MNM}, every maximal finite subgroup of $PSL_2(\calO_k)$ appears
as a stabilizer of some point in $\dbH\times\dbH$. We will say that two fixed
points in $\dbH\times\dbH$
are inequivalent if their isotropy groups are not conjugate. Thus we have a
bijection between conjugacy classes of maximal finite subgroups of 
$PSL_2(\calO_k)$ and the number of inequivalent fixed points in
$\dbH\times\dbH$. The problem of finding the number of fixed points of the action of the Hilbert modular group of a quadratic number field has been addressed
in \cite{prestel}. We sketch here the way to proceed: since only elliptic elements
of $PSL_2(\calO_k)$ can have fix points, then the trace of such an element
must satisfy
\beq
[Tr\: \sigma_i(h)]^2-4<0,\ \ \ \ \ i=1,2.
\eeq
Also, every elliptic matrix is conjugate (in $PSL_2(\dbC)$ ) to one of the form 
$\begin{pmatrix}
\omega & 0\\
0 & \overline{\omega}
\end{pmatrix}$
where $\omega$ is a primitive root of unity.
Thus, in our case, each of these traces should be an algebraic integer of 
$\dbQ(\sqrt{5})$, i.e. an element of $\dbZ[\frac{1+\sqrt{5}}{2}]$. The only possibilities are:
$0$, $\pm 1$ and $\pm\frac{1\pm\sqrt{5}}{2}$, corresponding to a 4th, 6th and 10th root of 
unity respectively. This says that in $PSL_2(\calO_k)$ we will only have elliptic elements 
of order $2,3$ and $5$, consequently the only finite subgroups that can appear will be isomorphic
to $\dbZ_2$, $\dbZ_3$ and $\dbZ_5$. The number of conjugacy classes of these maximal finite subgroups
is calculated in \cite{prestel}: each of these groups has exactly two conjugacy classes in $PSL_2(\calO_k)$.
Therefore, by Theorem \ref{philbert} we have
\beq
Wh_q(PSL_2(\calO_{\dbQ(\sqrt{5})}))=Wh_q(\dbZ_2)^2\oplus Wh_q(\dbZ_3)^2\oplus Wh_q(\dbZ_5)^2.
\eeq
In the special case of the classical Whithead group, i.e. $q=1$, it is known (See \cite{Ol88}) that $Wh(\dbZ_2)\simeq Wh(\dbZ_3)\simeq 0$ and $Wh(\dbZ_5)\simeq\dbZ$. Therefore
\beq
Wh(PSL_2(\calO_{\dbQ(\sqrt{5})}))\simeq\dbZ\oplus\dbZ.
\eeq
Also, by Corollary \ref{WhSL} we have
%\begin{align*}
\beq
Wh(SL_2(\mathcal{O}_{\dbQ(\sqrt{5}}))\simeq 
\dbZ\oplus\dbZ\oplus\dbZ_2,
\eeq
since $PSL_2(\calO_{\dbQ(\sqrt{5})})$ is a perfect group.
%&\cong \dbZ_2\oplus\dbZ\oplus\dbZ,
%\end{align*}
%where the last equation follows since $PSL_2(\calO_{\dbQ(\sqrt{5})})$ is a
%perfect group.\\
\par Finally for the higher $K$-theory we obtain (here $G=PSL_2(\calO_{\dbQ(\sqrt{5})})$):
\begin{align*}
\text{rk }K_q(\dbZ[G])-\text{rk }H_q(BG;\calK(\dbZ))&=
\begin{cases}
2r(\dbZ_2)+2r(\dbZ_3)+2r(\dbZ_5)&\hspace{-1.8cm}\text{ if }q>2\text{ and }q\equiv 1\text{ mod }4\\
2c(\dbZ_2)+2c(\dbZ_3)+2c(\dbZ_5)&\hspace{-1.8cm}\text{ if }q>2\text{ and }q\equiv 3\text{ mod }4\\
2(r(\dbZ_2)-q(\dbZ_2))+2(r(\dbZ_3)-q(\dbZ_3))\\
+2(r(\dbZ_5)-q(\dbZ_5))&\text{ if }q=1
\end{cases}\\
&=
\begin{cases}
4+4+6&\text{ if }q>2\text{ and }q\equiv 1\text{ mod }4\\
0+2+4&\text{ if }q>2\text{ and }q\equiv 3\text{ mod }4\\
0+0+2&\text{ if }q=1.
\end{cases}\\
&=
\begin{cases}
14&\text{ if }q>2\text{ and }q\equiv 1\text{ mod }4\\
6&\text{ if }q>2\text{ and }q\equiv 3\text{ mod }4\\
2&\text{ if }q=1.
\end{cases}
\end{align*}
\begin{remark}
It is worth noticing that the formula of Corollary \ref{WhSL} also provides a 
calculation for the Whitehead group of $SL_2(\dbZ)$:
\begin{align*}
Wh_1(SL_2(\dbZ))&\simeq Wh_1(PSL_2(\dbZ))\oplus PSL_2(\dbZ)^{ab}\oplus\dbZ_2\\
				 &\simeq Wh_1(\dbZ_2\ast\dbZ_3)\oplus\dbZ_6\oplus\dbZ_2\\
				 &\simeq \dbZ_6\oplus\dbZ_2,
\end{align*}
where we have used the fact that 
$Wh_1(\dbZ_2\ast\dbZ_3)\simeq Wh_1(\dbZ_2)\oplus Wh_1(\dbZ_3)\simeq 0$.
\end{remark}
%\par It is also worth mentioning that, since the abelianization of the classical
%modular group $PSL_2(\dbZ)$ is $\dbZ_{6}$ and 
%$Wh(PSL_2(\dbZ))=\dbZ_2\oplus\dbZ_3$, our results yield the following calculation
%\beq
%Wh(SL_2(\dbZ))=\dbZ_{6}\oplus\dbZ_2\oplus\dbZ_2\oplus\dbZ_3.
%\eeq
%%%%%%%%%%%%%%%%%%%%%%%%%%%%%%%%%%%%%%%%%%%%%%%%%%%%%%%%%%%%%%%%%%%%
%%%%%%%%%%%%%%%%%%%%%%%%%%%%%%%%%%%%%%%%%%%%%%%%%%%%%%%%%%%%%%%%%%%%
%%%%%%%%%%%%%%%%%%%%%%%%%%%%%%%%%%%%%%%%%%%%%%%%%%%%%%%%%%%%%%%%%%%%%
{\footnotesize\bibliographystyle{alpha} %harvard, unsrt, alpha
\bibliography{myblib}}

\begin{thebibliography}{JPLMVP11}

\bibitem[AO06]{AO06}
A.~Alves and P.~Ontaneda.
\newblock A formula for the {W}hitehead group of a three-dimensional
  crystallographic group.
\newblock {\em Topology}, 45(1):1--25, 2006.

\bibitem[Bas65]{bass}
Hyman Bass.
\newblock The {D}irichlet unit theorem, induced characters, and {W}hitehead
  groups of finite groups.
\newblock {\em Topology}, 4:391--410, 1965.

\bibitem[BFJPP00]{bianchi}
E.~Berkove, F.~T. Farrell, D.~Juan-Pineda, and K.~Pearson.
\newblock The {F}arrell-{J}ones isomorphism conjecture for finite covolume
  hyperbolic actions and the algebraic {$K$}-theory of {B}ianchi groups.
\newblock {\em Trans. Amer. Math. Soc.}, 352(12):5689--5702, 2000.

\bibitem[BHM93]{BHM}
M.~B{\"o}kstedt, W.~C. Hsiang, and I.~Madsen.
\newblock The cyclotomic trace and algebraic {$K$}-theory of spaces.
\newblock {\em Invent. Math.}, 111(3):465--539, 1993.

\bibitem[BJPP01]{BJPP01}
Ethan Berkove, Daniel Juan-Pineda, and Kimberly Pearson.
\newblock The lower algebraic {$K$}-theory of {F}uchsian groups.
\newblock {\em Comment. Math. Helv.}, 76(2):339--352, 2001.

\bibitem[BJPP02]{BJPP02}
E.~Berkove, D.~Juan-Pineda, and K.~Pearson.
\newblock A geometric approach to the lower algebraic {$K$}-theory of
  {F}uchsian groups.
\newblock {\em Topology Appl.}, 119(3):269--277, 2002.

\bibitem[Car80]{carter}
David~W. Carter.
\newblock Lower {$K$}-theory of finite groups.
\newblock {\em Comm. Algebra}, 8(20):1927--1937, 1980.

\bibitem[DKR11]{DKR11}
James~F. Davis, Qayum Khan, and Andrew Ranicki.
\newblock Algebraic {$K$}-theory over the infinite dihedral group: an algebraic
  approach.
\newblock {\em Algebr. Geom. Topol.}, 11(4):2391--2436, 2011.

\bibitem[DL98]{DL98}
James~F. Davis and Wolfgang L{\"u}ck.
\newblock Spaces over a category and assembly maps in isomorphism conjectures
  in {$K$}- and {$L$}-theory.
\newblock {\em $K$-Theory}, 15(3):201--252, 1998.

\bibitem[DL03]{DL03}
James~F. Davis and Wolfgang L{\"u}ck.
\newblock The {$p$}-chain spectral sequence.
\newblock {\em $K$-Theory}, 30(1):71--104, 2003.
\newblock Special issue in honor of Hyman Bass on his seventieth birthday. Part
  I.

\bibitem[DQR11]{DQR11}
James~F. Davis, Frank Quinn, and Holger Reich.
\newblock Algebraic {$K$}-theory over the infinite dihedral group: a controlled
  topology approach.
\newblock {\em J. Topol.}, 4(3):505--528, 2011.

\bibitem[FJ93]{FJ93}
F.~T. Farrell and L.~E. Jones.
\newblock Isomorphism conjectures in algebraic {$K$}-theory.
\newblock {\em J. Amer. Math. Soc.}, 6(2):249--297, 1993.

\bibitem[Fre90]{freitag}
Eberhard Freitag.
\newblock {\em Hilbert modular forms}.
\newblock Springer-Verlag, Berlin, 1990.

\bibitem[Gru08]{Gr08}
Joachim Grunewald.
\newblock The behavior of {N}il-groups under localization and the relative
  assembly map.
\newblock {\em Topology}, 47(3):160--202, 2008.

\bibitem[Hir73]{hirzebruch}
Friedrich E.~P. Hirzebruch.
\newblock Hilbert modular surfaces.
\newblock {\em Enseignement Math. (2)}, 19:183--281, 1973.

\bibitem[Jah09]{Ja09}
Bj{\o}rn Jahren.
\newblock Involutions on the rational {$K$}-theory of group rings of finite
  groups.
\newblock In {\em Alpine perspectives on algebraic topology}, volume 504 of
  {\em Contemp. Math.}, pages 189--202. Amer. Math. Soc., Providence, RI, 2009.

\bibitem[JPLMVP11]{JPLMP11}
D.~Juan-Pineda, J.-F. Lafont, S.~Millan-Vossler, and S.~Pallekonda.
\newblock Algebraic {$K$}-theory of virtually free groups.
\newblock {\em Proc. Roy. Soc. Edinburgh Sect. A}, 141(6):1295--1316, 2011.

\bibitem[JPML06]{JPM06}
Daniel Juan-Pineda and Silvia Millan-L{\'o}pez.
\newblock Invariants associated to the pure braid group of the sphere.
\newblock {\em Bol. Soc. Mat. Mexicana (3)}, 12(1):27--32, 2006.

\bibitem[JPML10]{JPM10}
Daniel Juan-Pineda and Silvia Millan-L{\'o}pez.
\newblock The {W}hitehead group and the lower algebraic {$K$}-theory of braid
  groups on {$\Bbb S^2$} and {$\Bbb R\Bbb P^2$}.
\newblock {\em Algebr. Geom. Topol.}, 10(4):1887--1903, 2010.

\bibitem[JPSS14]{JPS14}
Daniel Juan-Pineda and Luis~Jorge S{\'a}nchez~Salda{\~n}a.
\newblock On the ranks of the algebraic {$K$}-theory of hyperbolic groups.
\newblock {\em Bol. Soc. Mat. Mex. (3)}, 20(2):277--285, 2014.

\bibitem[KLR]{KLR14}
Holger Kammeyer, Wolfgang L{\"u}ck, and Henrik R{\"u}ping.
\newblock The {F}arrell-{J}ones conjecture for artbitrary lattices in virtually
  connected lie groups.
\newblock {\em arXiv:1401.0876v1}.

\bibitem[LO09]{LO09}
Jean-Fran{\c{c}}ois Lafont and Ivonne~J. Ortiz.
\newblock Lower algebraic {$K$}-theory of hyperbolic 3-simplex reflection
  groups.
\newblock {\em Comment. Math. Helv.}, 84(2):297--337, 2009.

\bibitem[LS00]{LS00}
Wolfgang L{\"u}ck and Roland Stamm.
\newblock Computations of {$K$}- and {$L$}-theory of cocompact planar groups.
\newblock {\em $K$-Theory}, 21(3):249--292, 2000.

\bibitem[Oli88]{Ol88}
Robert Oliver.
\newblock {\em Whitehead groups of finite groups}, volume 132 of {\em London
  Mathematical Society Lecture Note Series}.
\newblock Cambridge University Press, Cambridge, 1988.

\bibitem[Pea98]{pearson98}
Kimberly Pearson.
\newblock Algebraic {$K$}-theory of two-dimensional crystallographic groups.
\newblock {\em $K$-Theory}, 14(3):265--280, 1998.

\bibitem[Pre68]{prestel}
Alexander Prestel.
\newblock Die elliptischen {F}ixpunkte der {H}ilbertschen {M}odulgruppen.
\newblock {\em Math. Ann.}, 177:181--209, 1968.

\bibitem[vdG88]{geer}
Gerard van~der Geer.
\newblock {\em Hilbert modular surfaces}, volume~16 of {\em Ergebnisse der
  Mathematik und ihrer Grenzgebiete (3) [Results in Mathematics and Related
  Areas (3)]}.
\newblock Springer-Verlag, Berlin, 1988.

\bibitem[Wal78]{W78}
Friedhelm Waldhausen.
\newblock Algebraic {$K$}-theory of generalized free products. {I}, {II}.
\newblock {\em Ann. of Math. (2)}, 108(1):135--204, 1978.

\end{thebibliography}

\end{document}